\documentclass[a4paper,11pt]{article}
\usepackage[utf8]{inputenc}
\usepackage[margin=3cm]{geometry}
\usepackage{authblk}
\usepackage{color,colortbl}

\usepackage{amsmath,amssymb,amsthm}
\usepackage{xspace}
\usepackage{xstring}
\usepackage{graphicx}
\usepackage{cleveref}
\usepackage{tikz}
\usepackage{booktabs}
\usepackage{enumerate} 
\usepackage{siunitx}
\usepackage{subcaption}
\usepackage{url}
\usepackage{nicefrac}

\newtheorem{thm}{Theorem}
\newtheorem{prop}[thm]{Proposition}

\theoremstyle{definition}

\usepackage{pifont}
\newcommand{\cmark}{\ding{51}}

\usepackage[acronym,nomain,nonumberlist,automake,section=subsection]{glossaries}
\makeglossaries

\usepackage{todonotes}

\newcommand{\R}{\mathbb{R}}

\usepackage{natbib}
 \bibpunct[, ]{(}{)}{,}{a}{}{,}%

\begin{document}
\title{Driver-aware charging infrastructure design}

\author[1]{Stefan Kober}
\author[2]{Maximilian Schiffer}
\author[3]{Stephan Sorgatz}
\author[4]{Stefan Weltge}
\affil[1,4]{\textit{\small{Department of Mathematics, TUM School of Computation, Information and Technology, Technical University of Munich}}}
\affil[2]{\textit{\small{TUM School of Management \& Munich Data Science Institute, Technical University of Munich}}}
\affil[3]{\textit{\small{Volkswagen AG}}}
\date{}

\newacronym{cs}{CS}{charging station}
\newacronym{db-cslp}{DB-CSLP}{driver-based charging station location problem}
\newacronym{dvu}{DVU}{driver vehicle unit}
\newacronym{ev}{EV}{electric vehicle}
\newacronym{milp}{MILP}{mixed integer linear problem}
\newacronym{rslp-r}{RSLP-R}{refueling station location problem with routing}
\newacronym{soc}{SOC}{state of charge}

\newglossaryentry{actionNumber}{
	name=\ensuremath{n_a},
	description={number of trips/breaks of a given \gls{dvu}}
}
\newglossaryentry{agentSet}{
	name=\ensuremath{A},
	description={set of electric \glspl{dvu}}
}
\newglossaryentry{breakSequence}{
	name=\ensuremath{B_a},
	description={sequence of breaks of a given \gls{dvu}}
}
\newglossaryentry{chargingModes}{
	name=\ensuremath{M},
	description={set of charging modes}
}
\newglossaryentry{chargingPatterns}{
	name=\ensuremath{P_a},
	description={set of all charging patterns of a given \gls{dvu}}
}
\newglossaryentry{endTime}{
	name=\ensuremath{te^i_a},
	description={end time of a break}
}
\newglossaryentry{locationSet}{
	name=\ensuremath{L},
	description={set of potential charging station locations}
}
\newglossaryentry{modeNumber}{
	name=\ensuremath{n_M},
	description={number of different available charging modes}
}
\newglossaryentry{proximityFunction}{
	name=\ensuremath{\rho_a},
	description={proximity function that assigns each break a set of locations}
}
\newglossaryentry{startTime}{
	name=\ensuremath{ts^i_a},
	description={start time of a break}
}
\newglossaryentry{stationNumber}{
	name=\ensuremath{n_S},
	description={number of possible charging stations in one location}
}
\newglossaryentry{tripSequence}{
	name=\ensuremath{T_a},
	description={sequence of trips of a given \gls{dvu}}
}
\newglossaryentry{validityFunction}{
	name=\ensuremath{v_a},
	description={function that checks whether a pair of pattern and \gls{soc} is valid}
}
\newglossaryentry{validPatterns}{
	name=\ensuremath{VP_a},
	description={valid patterns of a given \gls{dvu}}
}	

\maketitle

\abstract{%
    Public charging infrastructure plays a crucial role in the context of electrifying the private mobility sector in particular for urban regions.
	Against this background, we develop a new mathematical model for the optimal placement of public charging stations for electric vehicles in cities.
	While existing approaches strongly aggregate traffic information or are only applicable to small instances,	we formulate the problem as a specific combinatorial optimization problem that incorporates individual demand and temporal interactions of drivers, exact positioning of charging stations, as well as various charging speeds, and realistic charging curves.
    We show that the problem can be naturally cast as an integer program that, together with different reformulation techniques, can be efficiently solved for large instances.
    More specifically, we show that our approach can compute optimal placements of charging stations for instances based on traffic data for cities with up to $600\,000$ inhabitants and future electrification rates of up to $15\%$.
}%

\section{Introduction}\label{sec:intro}

The promotion and deployment of electric vehicles is a key part of the transformation of the mobility sector towards more emission-free and sustainable transport.
One of the main challenges is the coordination of charging processes and mobility behavior in general.
For commercial fleets a multitude of studies and projects that support the recharging process exist.
Contrarily, in the private sector, early adopters mainly rely on private charging solutions, e.g., wallboxes on their own property.
Especially in urban environments, potential \gls{ev} users do not necessarily have a private parking space, such that the supply of public charging infrastructure is crucial to increase the acceptance of \glspl{ev}.
Accordingly, identifying the right amount of and locations for public charging stations in cities is of particular interest in order to achieve the goals for the market diffusion of \glspl{ev}.
In fact, governments have already set aggregated targets for the development of charging infrastructure.
However, municipalities are struggling with the implementation due to the high costs of installing charging stations and the inherent complexity of the problem, which stems from the fact that the infrastructure is addressed to the general public, i.e., obtaining good solutions requires to solve a large scale facility location problem.

Accordingly, designing charging station networks has gained significant interest in transport optimization and various approaches have been proposed to accomplish this planning task.
The most common approaches are node-based and aggregate charging demand without a temporal dimension, or are path-based and account for charging demands in aggregated flows. 
However, these existing approaches suffer from two major drawbacks: first, both approaches do not consider driver behavior via individual patterns. 
Second, as a consequence of aggregation, even path-based models insufficiently consider temporal interactions between drivers that occupy the same charging stations. 
Mitigating these fundamental drawbacks by incorporating individual and realistic demands and temporal interactions into a mathematical planning approach remains a fundamental challenge to provide profound decision support that enables the market diffusion of \glspl{ev}.
First approaches that tackle these issues can only solve small-scale instances or are simulation-based approaches that cannot be handled within rigorous optimization methods.

Against this background, we develop a new approach to compute optimal placements of public charging stations, incorporating the aforementioned driver-based features.
Moreover, our framework supports individual charging modes and curves, a flexible choice of objectives, and exact positioning.
Despite incorporating this level of complexity, we present an integer programming based approach that remains at the one hand computationally tractable and at the other hand easy to implement.
We demonstrate the effectiveness of our approach on instances based on traffic data of German major cities.
We show that our method scales reasonably well for future electrification rates of up to $15\%$ for cities of up to $600\,000$ inhabitants.

\subsection{Literature Review}
\label{secLiteratureReview}
In the following, we give a short overview over the related literature for the placement of public charging stations, which comprises three main streams, considering flow-based demand for recharging facilities \citep{hodgson_1990,kuby_2005}, arc-based demand \citep{capar_2013}, or node-based demand.
The flow-based and arc-based models find a cover of origin-destination-pairs, or arcs in a given network respectively.
In particular the \gls{rslp-r} is an active area of research for which recently specialized branch-and-cut~\citep{arslan_2019,gopfert_2019} as well as branch-and-price~\citep{yildiz_2016} approaches were developed.
In these studies, refueling stations are placed along long trips, which underlines their applicability for inter-regional travel.
While it is possible to adapt their approaches to an urban context by mapping refueling possibilities to pre-existing breaks, several problem characteristics that are native to the traffic in cities are not considered in this stream of research.
This includes for instance different charging speeds, the capacity of charging stations or the effect of break durations, i.e. that short stops may not be enough to fully recharge an \gls{ev}.

In contrast, node-based approaches aim to cover discrete demand (clusters) at given points.
Most node-based approaches work in a two-fold manner:
in a first step, a (tempo-)spatial demand is estimated. 
This is usually done by aggregating the demand of several vehicles and forming demand clusters.
These clusters are covered in a second step under consideration of problem-specific constraints, e.g., vehicle-to-grid technologies \citep{degennaro_2015}, the exact consumption \citep{andrenacci_2016,hidalgo_2016}, or the electric grid \citep{bayram_2013}.
In general, this approach is more appropriate for the urban context, but the static demand aggregation inherently neglects important features of the problem.
Especially dynamic features like individual user behavior or interaction of users with respect to time are hard to portray in such models \citep{adenaw_2020}.
There exist some attempts to include these dynamic properties in the literature, for instance \cite{cavadas_2015} introduce transferable demand and multiple periods, \cite{adenaw_2020} combine a simulation approach with an evolutionary process to find the best charging stations, \cite{shahraki_2015} accurately model individual demand for plug-in hybrid electric vehicles, and \cite{andrews_2013} introduce a MIP that accounts for user interaction.
As we will detail in Section~\ref{sec:problem}, these approaches leave room for a more detailed analysis of dynamic properties.

\subsection{Contribution}


We aim to derive a mathematical programming based approach that allows to determine positions of charging stations in an urban environment such that \gls{ev} drivers experience a convenient charging experience, i.e., can maintain their individual daily schedule without significant deviations for charging.
To this end, we formulate the problem of finding such positions in a purely combinatorial manner, accounting for individual driver patterns, and temporal interactions, as well as individual charging modes and curves, a flexible choice of objectives, and exact positioning.
This formulation leads to \gls{milp} models, for which we present strengthened formulations that allow to handle large instances and various objective functions.
As a byproduct, this approach also yields dual bounds on the quality of our placements.

We show how to utilize this approach in practice by making use of traffic data that contains information of individual drivers on a typical day within a given city, which can easily be gathered from various traffic simulation frameworks, e.g., MATSim~\citep{horni_2016}. 
We present results of a case study for the city of Düsseldorf (Germany, $\approx620$k citizens), which show that our approach allows to solve instances with the current electrification rate ($\approx1\%$), within few minutes and instances with electrification rates up to $15\%$ in few hours.
This equals placing up to $1500$ charging ports optimally to cover more than $3800$ driver profiles with recharging demand, which improves significantly upon the current state of the art approach that handles temporal interactions \citep{cavadas_2015}, which has been limited to placing $9$ stations for $300$ drivers.

Moreover, we evaluate the quality of the obtained solutions within a simulation environment to show the efficacy of our approach from a practitioners' perspective.

To foster future research and the use of our approach in practice, we open the implementation of our optimization method, our simulation environment, as well as the case study data on Github\footnote{\url{https://github.com/tumBAIS/driverAwareChargingInfrastructureDesign}}.

\subsection{Organization}

In \Cref{sec:problem}, we describe our general approach, introduce the most relevant objects that we deal with and explain how they are exploited to incorporate the aforementioned driver-based features.
We formulate the task of placing public charging stations as a combinatorial optimization problem in \Cref{sec:method}, which gives rise to a natural integer program.
Moreover, we discuss reformulation techniques that are crucial in order to solve real-world instances efficiently.
The second part of our paper is concerned with computational experiments.
In \Cref{sec:case-study}, we show how to derive a case study from existing traffic data.
\Cref{sec:computational-study} demonstrates the effectiveness of the proposed reformulations, and shows that the resulting model is capable to solve real-world instances.
We close our paper with remarks on possible applications and extensions of our work in \Cref{sec:conclusion}.
\section{Problem setting}\label{sec:problem}
In this section, we describe our general approach for placing public charging infrastructure in an urban environment.
We introduce the most relevant objects that we deal with and explain how they are exploited to incorporate the aforementioned driver-based features.

We assume that we have access to a set of \emph{drivers} $D$ (of electric vehicles) and their schedules on a representative day (or set of days).
Here, we identify each driver with her vehicle, typically referred to as \gls{dvu}.
For each driver $d\in D$, we are given a list of \emph{trips} $T(d)$, where each trip is defined by its start location, end location, start time, end time, and the amount by which the driver's \gls{soc} decreases during the trip.
Implicitly, this yields information about stops in between trips, which, for example, correspond to shorter stays while shopping or longer stays while working, and hence to potential charging activities that do not affect a driver's actual schedule.
Moreover, information about charging characteristics of $d$ is available.
We note that, while one may argue that this driver behavior is hardly deterministic over a longer time horizon, we will see in \Cref{sec:computational-study} that it is still possible to derive representative scenarios for strategic planning.

\paragraph{Individual charging modes and curves}
For every \emph{charging mode} $m\in M$, we are given a \emph{charging curve} $f_{m,d}$.
The set $M$ refers to different charging modes and typically consists of two modes: AC charging (slow) and DC charging (fast).
The function $f_{m,d}$ determines, for given $x$ and $t$, the \gls{soc} of $d$ after charging in mode $m$ for $t$ units of time, starting with an \gls{soc} of $x$.

\paragraph{Individual demands}
We are further given each driver's \gls{soc} at the beginning of the day as well as bounds within which the \gls{soc} must stay over the course of the entire day.
Every trip decreases the driver's \gls{soc} by a given certain amount, depending on the trip.
We assume that the \gls{soc} can only be increased by charging at public charging stations, which results in a recharging demand for each driver.
Note, that our planning problem only encompasses the placement of public charging infrastructure.
Private charging possibilities can still be portrayed, for instance by adapting the \gls{soc} and the bounds for a driver that has access to home charging.
Consequently, we also omit commercial fleets from our model, since they usually charge at specially dedicated charging stations and therefore do not need to rely on public charging infrastructure.

\paragraph{Accurate placements}
In order to allow for charging operations, public charging stations must be made available at a subset of \emph{potential locations} $L$.
The latter is a finite set of locations that may arbitrarily arise from the operating area.
Each location can be equipped with a certain type of charging station, which is defined by a number of charging ports of a specific mode.
To obtain a high degree of accuracy in the placement, the set of locations has to be defined accordingly. 
This is for instance possible by using a grid approach, similar to the one employed by \cite{cavadas_2015}.

\paragraph{Capacities and time interactions}
Once charging stations have been placed, drivers may occupy single charging ports between their trips.
Since time intervals between two consecutive trips typically refer to activities carried out by the drivers, we assume that drivers charge at (and hence occupy) a port for the whole time interval.
At every time at most one driver may be connected to a port.

\paragraph{Different objectives}
Our goal is to determine a subset of locations $L'\subseteq L$, and for each location $\ell\in L'$ the type of charging station installed at $\ell$.
A possible objective is to minimize the total cost associated to the selected charging stations under the constraint that every driver's demand can be satisfied.
In this setting, a driver may only be allowed (or willing) to occupy charging ports at locations whose distance to the driver's current location is below a certain threshold.
However, as we will see later, our approach can be easily adapted to various other objectives.

\bigskip
\noindent

\Cref{tab:literature_model_features} shows the problem characteristics of the closest related works.
As can be seen, existing approaches that are node or path-based are partially scalable, but neglect individual demand, realistic charging modes and curves, as well as temporal interactions.
First approaches that aim to mitigate these shortcomings partially consider capacities but neglect individual demand and multiple objectives or consider the latter, but ignore capacities.
All of these approaches are not scalable to large instances.
Concluding, to the best of the authors' knowledge, none of the existing approaches accounts for individual driver patterns, temporal interactions, individual charging modes and curves, a flexible choice of objectives, as well as exact positioning, and is at the same time scalable to large instances.

\begin{table}
    \noindent\makebox[\textwidth]{
    \footnotesize
	\begin{tabular}{ p{40mm} | c | c | c | c | c}
        & \rotatebox[]{280}{\parbox{22mm}{\mbox{charging station} capacities \& interactions}} & \rotatebox[]{280}{\parbox{22mm}{individual demand}} & \rotatebox[]{280}{\parbox{22mm}{charging modes \& curves}} & \rotatebox[]{280}{\parbox{22mm}{different \mbox{objectives}}} & \rotatebox[]{280}{\parbox{22mm}{scalability}}\\
		\hline
        \hline
        \cite{hodgson_1990,bayram_2013} & & & & & \cmark \\
        \hline
        \cite{kuby_2005,gopfert_2019,yildiz_2016} & &  & & & \cmark \\
        \hline
        \cite{capar_2013,arslan_2019} & &  & & \cmark & \cmark \\
        \hline
        \cite{hidalgo_2016} & & \cmark & (\cmark) & \cmark & \\
        \hline
        \cite{cavadas_2015} & (\cmark) & & & & \\
        \hline
        \cite{shahraki_2015} & & \cmark & \cmark & & (\cmark) \\
        \hline
        \cite{andrews_2013} & \cmark & & (\cmark) & & \\
        \hline
		Our work & \cmark & \cmark & \cmark & \cmark & \cmark \\
    \end{tabular}
    }
    \caption{Characteristics of previous optimization-based approaches for the placement of public charging infrastructure.}
    \label{tab:literature_model_features}
\end{table}

\section{Method}\label{sec:method}
In this section, we derive a mathematical formulation for the task of positioning public charging stations that is based on the setting described in the previous section.
We first formulate an optimization problem that includes all relevant objects and constraints in a combinatorial manner.
This formulation gives rise to a natural integer program, which we discuss in \Cref{sec:ip}. 
In \Cref{sec:enhance}, we describe modifications to the integer program that do not affect the set of feasible solutions but are crucial in order to solve real-world instances efficiently.

\subsection{Combinatorial optimization problem}\label{sec:cop}
Recalling \Cref{sec:problem}, we assume that we are given finite sets of drivers $D$, charging modes $M$, and locations for potential charging stations $L$.
Moreover, we are given a set of trips $T(d)$ for every driver $d\in D$.

\paragraph{Breaks}
Since we associate the time intervals between consecutive trips with a driver's charging opportunities, we extract a set of \emph{breaks} $B(d)$ from the trip data for each driver.
A break is characterized by its start time, end time, and location.
We assume that breaks are individual objects for each driver, i.e., $B(d)\cap B(d')=\emptyset$ for $d\neq d'\in D$.

\paragraph{Nearby locations}
For each break $b\in B(d)$, we compute a set of \emph{nearby locations} $L(b)\subseteq L$ that driver $d$ is willing to charge at during break $b$.
For example, the set $L(b)$ may consist of all locations in $L$ that are within a certain distance to the location of $b$.
However, we do not make any specific assumptions on $L(b)$ in order to keep our formulation as general as possible.

\paragraph{Feasible charging plans}
To identify at which breaks a driver $d\in D$ is charging (in a specific mode), we define a \emph{charging plan} of $d$ as a set of pairs $(b,m)\in B(d)\times M$.
We say that a charging plan $P\subseteq B(d)\times M$ is \emph{feasible} if its charging operations ensure that $d$ can complete all trips while respecting the pre-defined bounds on the \gls{soc} at any time.
Given the charging curves of $d$, we assume that we can efficiently determine whether a charging plan is feasible.
We denote the set of all feasible plans of a driver~$d$ by $F(d):=\{P\subseteq B(d)\times M: P\text{ is feasible}\}$.

\paragraph{Charging stations}
We consider a \emph{charging station} as a tuple $(\ell,m,\Delta)\in L\times M\times\mathbb{Z}_{\ge1}$ indicating its location $\ell$, mode $m$, and number of ports $\Delta$.
We assume that we are given a finite set $\Omega\subseteq L\times M\times\mathbb{Z}_{\ge1}$ denoting possible charging stations.
Each charging station is associated with a (possibly individual) \emph{cost} $c:\Omega\to\mathbb{R}_{\ge0}$.

\paragraph{Time points}
Finally, we define a set of \emph{time points} $T$ in order to monitor the capacity utilization of charging stations at these time points.
We use the shorthand notation $t\in b$ to denote that a time point $t\in T$ is contained in the time interval of the break $b\in B(d)$.

\paragraph{Task}
Our goal is to find a set of charging stations $S\subseteq\Omega$ that minimizes the total cost $\sum_{s\in S}c(s)$ under the following constraints.
\begin{enumerate}
    \item For each driver $d\in D$ there must exist a feasible charging plan $P(d)\in F(d)$, where every charging process $(b,m)\in P(d)$ is assigned to a charging station $s(b,m)=(\ell,m,\Delta)\in S$ with $\ell\in L(b)$.
    \item Moreover, these assignments have to respect the charging station capacities at any time, i.e.  
    \begin{equation*}
        |\{(b,m)\in P(d):d\in D,\,s(b,m)=s,\,t\in b\}|\le\Delta
    \end{equation*}
    holds for every $t\in T$ and $s=(\ell,m,\Delta)\in S$.
\end{enumerate}

We note that it is possible to adjust the constraints and objective, for instance by introducing a budget constraint and maximizing the amount of satisfied demand, or by working on several driver sets in parallel.

\subsection{Integer program}\label{sec:ip}
In this section, we derive an integer programming formulation of the above combinatorial optimization problem.
First, we introduce a binary variable for every charging station in $\Omega$, i.e., 
\[x_{\ell,m,\Delta}\in\{0,1\}\quad\forall\ (\ell,m,\Delta)\in\Omega,\]
where $S$ corresponds to all $(\ell,m,\Delta)$ with $x_{\ell,m,\Delta}=1$.
The objective function is easily expressed as
\[\mathrm{minimize}\ \sum_{(\ell,m,\Delta)\in\Omega}c(\ell,m,\Delta)\cdot x_{\ell,m,\Delta}.\]
The constraints
\[\sum_{\substack{m,\Delta:\\(\ell,m,\Delta)\in\Omega}}x_{\ell,m,\Delta}\le1\quad\forall\ \ell\in L\]
ensure that we can create at most one charging station per location.
In order to capture charging operations of drivers, we introduce binary variables 
\[y_{d,b,\ell,m}\in\{0,1\}\quad\forall\ d\in D,\,b\in B(d),\,\ell\in L(b),\,m\in M,\]
where $y_{d,b,\ell,m}=1$ encodes that driver $d$ charges at location $\ell$ with mode $m$ during break $b$.
Recall that a charging station $(\ell,m,\Delta)$ can be occupied by at most $\Delta$ drivers simultaneously, which is expressed by the constraints
\begin{equation}\label{eq:cap_constrs}
    \sum_{d\in D}\sum_{\substack{b\in B(d):\\\ell\in L(b),\,t\in b}}y_{d,b,\ell,m}\le\sum_{\substack{\Delta:\\(\ell,m,\Delta)\in\Omega}}\Delta\cdot x_{\ell,m,\Delta}\quad\forall \ \ell\in L,\,m\in M,\,t\in T.
\end{equation}
Notice that for fixed $\ell^*\in L$ the sets $\{(d,b):d\in D,\,b\in B(d),\,\ell^*\in L(b),\,t\in b\}$ can be identical for different time points $t\in T$, resulting in duplicate constraints.
In fact, it suffices to impose the above constraints for the start times of all breaks $b$ with $\ell^*\in L(b)$ only.

Finally, we need to make sure that every driver charges according to one of her feasible charging plans.
A straightforward way to model this is by introducing binary variables
\begin{equation*}
    z_{d,P}\in\{0,1\}\quad\forall\ d\in D,\,P\in F(d),
\end{equation*}
where $z_{d,P}=1$ encodes that driver $d$ charges according to $P$.
The requirement that every driver has to decide on exactly one feasible charging plan is expressed as 
\begin{equation*}
    \sum_{P\in F(d)}z_{d,P}=1\quad\forall\ d\in D.
\end{equation*}
The constraints
\begin{equation}\label{eq92437d}
    \sum_{\ell\in L(b)}y_{d,b,\ell,m}=\sum_{\substack{P\in F(d):\\(b,m)\in P}}z_{d,P}\quad\forall\ d\in D,\,b\in B(d),\,m\in M
\end{equation}
ensure that driver $d$ selects charging operations that match the chosen plan.

Note that the model variations mentioned in~\Cref{sec:cop} can be captured in a similar way.

\subsection{Strengthened formulations}\label{sec:enhance}
It turns out that the above integer programming formulation can be significantly improved since its linear programming relaxation is rather weak in terms of its integrality gap.
That is, it admits non-integer points satisfying all linear constraints but whose objective value is much smaller than the true optimum value.
Fortunately, it is possible to derive additional linear inequalities that reduce the integrality gap significantly but do not affect the set of feasible integer solutions.
We present a family of such cuts in \Cref{subsubsec:capcuts}.

Moreover, we discuss the constraints in our model that force drivers to follow feasible charging plans in \Cref{sec:cpc}.
We observe that they often can be reformulated by eliminating several variables, which has an additional positive impact on the computation time needed to solve our model.
The effectiveness of the strengthened formulations is demonstrated in \Cref{sec:computational-study}.

\subsubsection{Capacity cuts}\label{subsubsec:capcuts}
To illustrate why our original formulation is weak, suppose that an (optimal) solution to the integer program decides to build a charging station $(\ell^*,m^*,\Delta)\in\Omega$ with $\Delta=k\ge2$ ports.
Suppose further that at every time at most one driver is charging at this station.
This means that the left-hand sides of constraints~\eqref{eq:cap_constrs} that correspond to $\ell^*,m^*$ are at most $1$ for all $t\in T$.
Since these constraints are the only ones that relate $x$ and $y$ directly, we can set $x_{\ell^*,m^*,\Delta}=\frac1k$ and still obtain a solution that is feasible to the linear programming relaxation.
Clearly, setting a variable to $\frac1k$ instead of $1$ may decrease the objective value significantly.
A simple way to exclude such fractional points is by adding the inequalities 
\begin{equation}\label{eqajad32}
    y_{d,b,\ell^*,m^*}\le\sum_{\Delta:(\ell^*,m^*,\Delta)\in\Omega} x_{\ell^*,m^*,\Delta}
\end{equation}
for all $d\in D$ and $b\in B(d)$ with $\ell^*\in L(b)$, which are certainly valid for all integer solutions.
In our computational study, we show that adding these inequalities improves our model significantly.

We observe that, in the case where there is only one $\Delta$ with $(\ell^*,m^*,\Delta)\in\Omega$, inequalities~\eqref{eq:cap_constrs} together with~\eqref{eqajad32} cannot be further strengthened in the following sense:
\begin{prop}\label{prop1}
    Let $\mathcal{B}:=\{b:d\in D,\,b\in B(d),\,\ell^*\in L(b)\}$ and suppose that $(\bar y,\bar x)\in\R^\mathcal B_{\ge0}\times[0,1]$ satisfies $\sum_{b\in\mathcal{B}:t\in b}\bar y_b\le\Delta\bar x$ for all $t\in T$ and $\bar y_b\le\bar x$ for all $b\in\mathcal{B}$. 
    Then $(\bar y,\bar x)$ is a convex combination of binary vectors $(y',x')\in\{0,1\}^{\mathcal{B}}\times\{0,1\}$ that satisfy $\sum_{b\in\mathcal{B}:t\in b}y'_b\le\Delta x'$ for all $t\in T$.    
\end{prop}
\begin{proof}
    If $\bar x = 0$, then this implies $\bar y = \mathbf{0}$ in which case we are done.
    Otherwise, let $y^* := \frac{1}{\bar x} \bar y$ and observe that $(\bar y,\bar x) = (1-\bar x) (\mathbf{0},0) + \bar x (y^*,1)$, that is, $(\bar y,\bar x)$ is a convex combination of $(\mathbf{0},0)$ and $(y^*,1)$.
    Thus, it remains to show that $y^*$ is a convex combination of binary vectors $y' \in \{0,1\}^{\mathcal{B}}$ that satisfy $\sum_{b\in\mathcal{B}:t\in b}y'_b\le\Delta$ for all $t\in T$.

    To this end, notice that $y^*$ satisfies
    $
        \sum_{b\in\mathcal{B}:t\in b} y^*_b = \frac{1}{\bar x}\sum_{b\in\mathcal{B}:t\in b}\bar y_b\le\Delta
    $
    for all $t\in T$ and $y^*_b=\frac{1}{\bar x}\bar y_b\le1$ for all $b\in\mathcal{B}$.
    In other words, $y^*$ is contained in the polytope
    \[
        Q := \left\{ \tilde y \in [0,1]^\mathcal B : \sum_{b\in\mathcal{B}:t\in b} \tilde y_b \le \Delta \text{ for all }t \in T\right\}.
    \]
    By ordering the constraints according to their time $t\in T$, we see that $Q$ can be written as $\{ y : Ay \le h, \, \mathbf{0}\le y \le \mathbf{1} \}$, where $h$ is an integer vector and $A$ is a $0/1$-matrix that has the \emph{consecutive ones property}.
    Such matrices are known to be totally unimodular~\citep{fulkerson_1965}, which implies that all vertices of $Q$ are integer, cf.~\cite[Theorem~19.3]{schrijver_1998}, and hence binary.
    Thus, $y^*$ is a convex combination of binary vectors $y' \in \{0,1\}^{\mathcal{B}} \in Q$ as claimed.
\end{proof}
However, even with these constraints there are still several constellations in which the $x$-values can be decreased in the linear programming relaxation.
For this reason, one may consider the following natural generalization of inequalities~\eqref{eqajad32}.
For every $t^*\in T$ and every set $S\subseteq\{(d,b):d\in D,\,b\in B(d),\,\ell^*\in L(b),\,t^*\in b\}$, consider the inequality
\begin{equation}\label{eq7u3njg}
    \sum_{(d,b)\in S}y_{d,b,\ell^*,m^*}\le\sum_{\Delta:(\ell^*,m^*,\Delta)\in\Omega}\min(|S|,\Delta)\cdot x_{\ell^*,m^*,\Delta}.
\end{equation}
Note that the inequalities~\eqref{eqajad32} arise as a special case from~\eqref{eq7u3njg} where $|S|=1$.
Again, it is easy to see that every integer solution satisfies~\eqref{eq7u3njg}.
Still, the inequalities~\eqref{eq7u3njg} do not suffice to generalize the observation of \Cref{prop1} to the case where there is more than one $\Delta$ with $(\ell^*,m^*,\Delta)\in\Omega$, and it is not clear how the inequalities can be further strengthened.
Moreover, given the strength of the inequalities in~\eqref{eqajad32} and the large number of the inequalities in~\eqref{eq7u3njg}, incorporating the latter inequalities in an efficient way remains challenging.

\subsubsection{Charging plan cuts}\label{sec:cpc}
Recall that we require every driver $d\in D$ to follow a feasible charging plan in $F(d)$.
To this end, we introduced auxiliary variables $z_{d,P}$ for all $P\in F(d)$ and constraints~\eqref{eq92437d} to link $y$ and $z$.
While this formulation is very explicit, we describe a reformulation that avoids the auxiliary variables $z$.

For a charging plan $P\in F(d)$ let $\chi(P)\in\{0,1\}^{B(d)\times M}$ denote its characteristic vector, i.e., $\chi(P)_{b,m}=1$ if and only if $(b,m)\in P$.
Consider the polytope $Q(d):=\operatorname{conv}\{\chi(P):P\in F(d)\}$.
Suppose we have computed an inequality description $Q(d)=\{q\in\R^{B(d)\times M}:Aq\ge h\}$ for $Q(d)$, where $A$ is a matrix and $h$ is a vector such that the rows of $A$ and $h$ correspond to some index set $I$ and the columns of $A$ correspond to $B(d)\times M$.
The following statement shows that the auxiliary variables $z$ and the constraints~\eqref{eq92437d} can be avoided by imposing some specific linear inequalities on $y$ only.
\begin{prop}\label{prop2}
    For a vector $y\in\R^\Gamma$ with $\Gamma:=\{(b,\ell,m):b\in B(d),\,\ell\in L(b),\,m\in M\}$, the following are equivalent.
    \begin{itemize}
        \item There exist $z\in[0,1]^{F(d)}$ with $\sum_{P\in F(d)}z_P=1$ and 
        \[
            \sum_{\ell\in L(b)}y_{b,\ell,m}=\sum_{\substack{P\in F(d):\\(b,m)\in P}}z_P 
        \]
        for all $b\in B(d)$ and $m\in M$.
        \item The linear inequalities 
        \begin{equation}\label{eqhf82t2}
            \sum_{(b,m)\in B(d)\times M}\left(\sum_{\ell\in L(b)}A_{i,(b,m)}y_{b,\ell,m}\right)\ge h_i    
        \end{equation}
        are satisfied for all $i\in I$.
    \end{itemize}
\end{prop}
\begin{proof}
    Let $y\in\R^\Gamma$, and suppose there is some $z\in[0,1]^{F(d)}$ satisfying the first condition.
    For every $i\in I$, we have
    \begin{align*}
        \sum_{(b,m)\in B(d)\times M}\left(\sum_{\ell\in L(b)}A_{i,(b,m)}y_{b,\ell,m}\right)&=\sum_{(b,m)\in B(d)\times M}A_{i,(b,m)}\sum_{\substack{P\in F(d):\\(b,m)\in P}}z_P\\
        &=\sum_{(b,m)\in B(d)\times M}A_{i,(b,m)}\sum_{P\in F(d)}z_P\chi(P)_{b,m}\\
        &=\sum_{P\in F(d)}z_P\sum_{(b,m)\in B(d)\times M}A_{i,(b,m)}\chi(P)_{b,m}\\
        &\ge\sum_{P\in F(d)}z_Ph_i=h_i,
    \end{align*}
    where the inequality holds since $\chi(P)$ is contained in $Q(d)$ and hence satisfies $A\chi(P)\ge h$ for all $P\in F(d)$.
    
    Suppose now that $y\in\R^\Gamma$ satisfies~\eqref{eqhf82t2} for all $i\in I$.
    Define $q\in\R^{B(d)\times M}$ via $q_{b,m}=\sum_{\ell\in L(b)}y_{b,\ell,m}$ for all $b\in B(d)$ and $m\in M$.
    For every $i\in I$ we have 
    \begin{align*}
        \sum_{(b,m)\in B(d)\times M}A_{i,(b,m)}q_{b,m}=\sum_{(b,m)\in B(d)\times M}\left(\sum_{\ell\in L(b)}A_{i,(b,m)}y_{b,\ell,m}\right)\ge h_i
    \end{align*}
    and hence $Aq\ge h$.
    This means, that $q$ is contained in $Q(d)$ and hence can be written as $q=\sum_{P\in F(d)}z_P\chi(P)$ for some $z\in[0,1]^{F(d)}$ with $\sum_{P\in F(d)}z_P=1$.
    For every $b\in B(d)$ and $m\in M$, we obtain 
    \[
        \sum_{\ell\in L(b)}y_{b,\ell,m}=q_{b,m}=\sum_{P\in F(d)}z_P\chi(P)=\sum_{\substack{P\in F(d):\\(b,m)\in P}}z_P.
    \]
\end{proof}
In fact, we suggest to replace the auxiliary variables $z$ and the constraints~\eqref{eq92437d} by the inequality description of $Q(d)$, whenever the latter can be easily computed.
For a general set of feasible charging plans $F(d)$ it is very difficult to derive an inequality description of $Q(d)$ in a closed form.
However, one may use existing software tools to compute the convex hull of a given set of points.
This is particularly feasible when the ambient dimension of $Q(d)$ is small.
Since the ambient dimension of $Q(d)$ is equal to $|B(d)|\cdot|M|$, this approach is applicable for drivers $d\in D$ with a reasonably small number of breaks.
Indeed, in our computational experiments, for most drivers the set $B(d)$ consists of at most five breaks, in which case the inequality description of $Q(d)$ can be quickly computed.
Moreover, it turns out that this reformulation improves our model further.
\section{Case study}\label{sec:case-study}

In this section, derive a case study from existing traffic data.
We perform our evaluations on open traffic data\footnote{\url{https://svn.vsp.tu-berlin.de/repos/public-svn/matsim/scenarios/countries/de/duesseldorf/}} that has been generated for the city of Düsseldorf, Germany ($\approx619\,000$ inhabitants in 2019) provided by~\cite{rakow_2021}.
Based on current traffic information, the authors used MATSim to simulate a collection of drivers together with their stops within a typical day in Düsseldorf.
In what follows, we describe how to turn such information into driver data as needed for our method.

Our data contains $268\,110$ individual agents and their mobility plans on a typical work day.
The set of agents encompasses all agents that spend a certain amount of time in Düsseldorf during the day.
Out of these agents, we identify $113\,852$ as actual residents, i.e., agents that start their first trip within the city borders of Düsseldorf, which are responsible for $477\,277$ trips, using different modes of transport.
The true number of such trips was approximately $2.16$ million in 2019 (calculations based on~\cite[]{gerike_2020}).
Therefore, we estimate that our set of agents and consequently the portrayed traffic represents approximately one quarter of the true size.

Extracting only the trips that were done by car, we obtain $475\,639$ trips that are done by $128\,059$ different drivers.
Recall, that this only represents approximately one quarter of the traffic, so we assume that our groundset of potential \glspl{ev} consists of $N := 512\,236$ drivers.
To identify a set of potential \gls{ev} drivers, we restrict ourselves to drivers whose consecutive trips end and start at locations within a distance of $300\mathrm{m}$, which we also require for the last and first trip to hold.

An \gls{ev} in our setting is characterized by its battery capacity, its range and its charging behavior.
Since we are investigating an urban environment, we assume all cars to be rather compact with a battery capacity of $50\mathrm{kWh}$ and an effective range of $260\mathrm{km}$, corresponding to a consumption of $19.23\mathrm{kWh}$ per $100\mathrm{km}$.
However, we remark that our approach would also allow to model arbitrary \glspl{ev} and assign them to different drivers.
We consider two modes of charging, AC and DC.
For AC charging, the charging speed can be assumed to be almost independent of the \gls{soc} of the battery. 
Therefore, we assume a constant charging speed of $11\mathrm{kW}$ and an efficiency of $0.85$, which results in an effective charging speed of $9.35\mathrm{kW}$.
The charging speed of DC charging typically follows a charging curve whose speed decreases with the \glspl{soc}.
We use a piecewise linear function as a charging curve, as depicted in~\Cref{fig:cc}.

\begin{figure}
    \begin{center}
        \includegraphics[width=.65\textwidth]{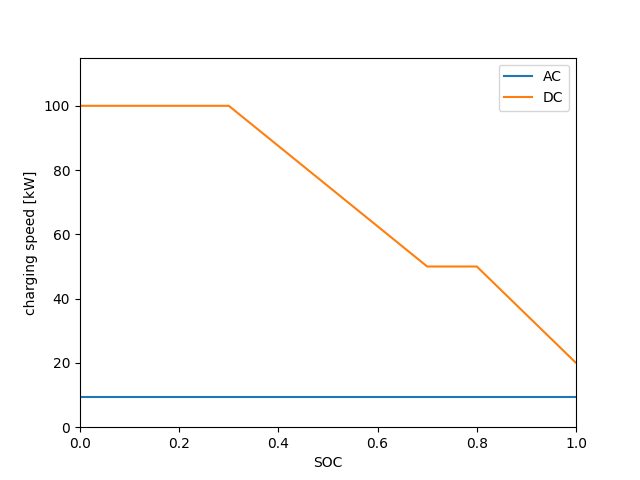}
    \end{center}
    \caption{Effective charging speeds of AC and DC charging stations in our case study.}
    \label{fig:cc}
\end{figure}

We assume that during an average day, drivers avoid their state of charge to fall below $10\%$ of its capacity.
Based on this, we exclude all drivers whose \gls{soc} will fall below $10\%$ even when starting with an \gls{soc} of $100\%$ in the morning and DC charging at every break within the planning region.
This results in a set of $100\,854$ remaining drivers, which we use as our final set of drivers $\mathcal{D}$ to sample from.
Out of these drivers, $26\,740$ live within the city of Düsseldorf and we denote this set by $\mathcal{D}_c \subseteq \mathcal{D}$.

We assume that all drivers from $\mathcal{D} \setminus \mathcal{D}_c$ have a wallbox, i.e., they have the opportunity to recharge their vehicle at home.
This is motivated by the facts that drivers in $\mathcal{D} \setminus \mathcal{D}_c$ live outside of our planning region and that houses in rural regions have a higher probability of being able to install a wallbox \cite[]{dena_2020}.
Moreover, we assume that $39\%$ of the drivers in $\mathcal{D}_c$, selected uniformly at random, also have access to a wallbox.
This percentage is based on the \emph{dena} study by~\cite{dena_2020}.

To generate an actual instance for our approach, we proceed as follows.
Assuming a specific electrification rate $r \in [0,1]$, we uniformly sample $r N \cdot \nicefrac{|\mathcal{D}_c|}{|\mathcal{D}|}$ distinct drivers from $\mathcal{D}_c$ and $r N \cdot (1 - \nicefrac{|\mathcal{D}_c|}{|\mathcal{D}|})$ drivers from $\mathcal{D} \setminus \mathcal{D}_c$.

We assume that all sampled drivers with a wallbox start with an \gls{soc} of $100\%$.
For each driver $d$ without a wallbox, let $\mu_d \ge 10\%$ be the minimum starting \gls{soc} needed to never reach an \gls{soc} below $10\%$ when DC charging at every break within the planning region.
The starting \gls{soc} of $d$ is then chosen uniformly at random from the interval $[\max(\mu_d,20\%),100\%]$.

We define a charging plan for a driver $d$ as feasible if it ensures that the \gls{soc} never falls below $10\%$ and reaches at least $\max(\mu_d,20\%)$ after the last break.
To this end, we only compute minimally feasible plans, i.e., those plans that do not satisfy the previous criteria when switching from a charging break to a non-charging break, or from a DC charging break to an AC charging break.
While calculating all minimal feasible charging plans is only tractable for a small number of breaks, we further declare charging plans to be feasible if they consist of at least four charging breaks.

To define the possible locations of charging stations, we use a grid approach (cf. \cite{cavadas_2015,csonka_2017}) by overlaying the entire city region of Düsseldorf with a square grid with cells of side length $100\mathrm{m}$, resulting in a total of $21\,745$ possible locations.
For each location, we allow $n_{\mathrm{AC}} \in \{2,4,6,8\}$ AC charging ports, or $n_{\mathrm{DC}} \in \{ 4,6,8\}$ DC charging ports, and define its cost to be $n_{\mathrm{AC}}$ or $2n_{\mathrm{DC}}$, respectively.

Finally, for each break $b$ of a driver, we define its set of nearby locations $L(b)$ as all locations within the grid whose distance is at most $200$ meters.
In order to reduce the size of the instance, we exclude those locations $\ell$ for which there exists another location $\ell'$ with $\{b\in B(d):d\in D,\,\ell\in L(b)\} \subseteq \{b\in B(d):d\in D,\,\ell'\in L(b)\}$.

\section{Results}\label{sec:computational-study}

In this section, we demonstrate the effectiveness of the strengthened formulations described in \Cref{sec:enhance} and show that the resulting model is capable to solve real-world instances.

\subsection{Effectiveness of strengthened formulations}
To analyze the effectiveness of the strengthened formulations presented in \Cref{sec:enhance}, first recall that we proposed to add the \emph{capacity cuts}~\eqref{eqajad32}.
Second, we suggested to use \emph{charging plan cuts}, i.e., inequality descriptions of $Q(d)$ for each driver $d$ (see~\Cref{sec:cpc}) instead of the auxiliary variables $z$ and constraints~\eqref{eq92437d}.
A third enhancement that turned out to be very efficient is to relax the binary variables $y$ (which assign charging processes to locations) to \emph{fractional} variables in $[0,1]$.
Clearly, this may result in selections of charging stations that do not permit a feasible assignment of charging processes to locations.
However, in all our experiments, we found that this caused almost no change in the value of the optimal solution, but had a significant impact on the performance of our model.

To evaluate the impact of the proposed enhancements, for each electrification rate in $\{1\%,2\%,\dots,6\%\}$, we independently sampled driver data for $10$ days as described in the previous section.
The size of the instances in our experiments spans a range between $200$ and $1\,200$ drivers (that need to charge at a public charging station) who have between $1\,000$ and $6\,000$ stops in total.
For each subset of the three enhancements, we ran the adapted model on all generated instances using \cite{gurobi} on a standard laptop. 
First, we evaluated the computational times, averaged over the ten instances per electrification rate, until Gurobi determines a gap of at most $1\%$.
Here, the gap is defined as $(P-D)/D$, where $P$ denotes the objective value of the best known feasible solution and $D$ denotes the dual bound.
The results are depicted in \Cref{fig:runtime}.

\begin{figure}
    \begin{center}
        \includegraphics[width=.8\textwidth]{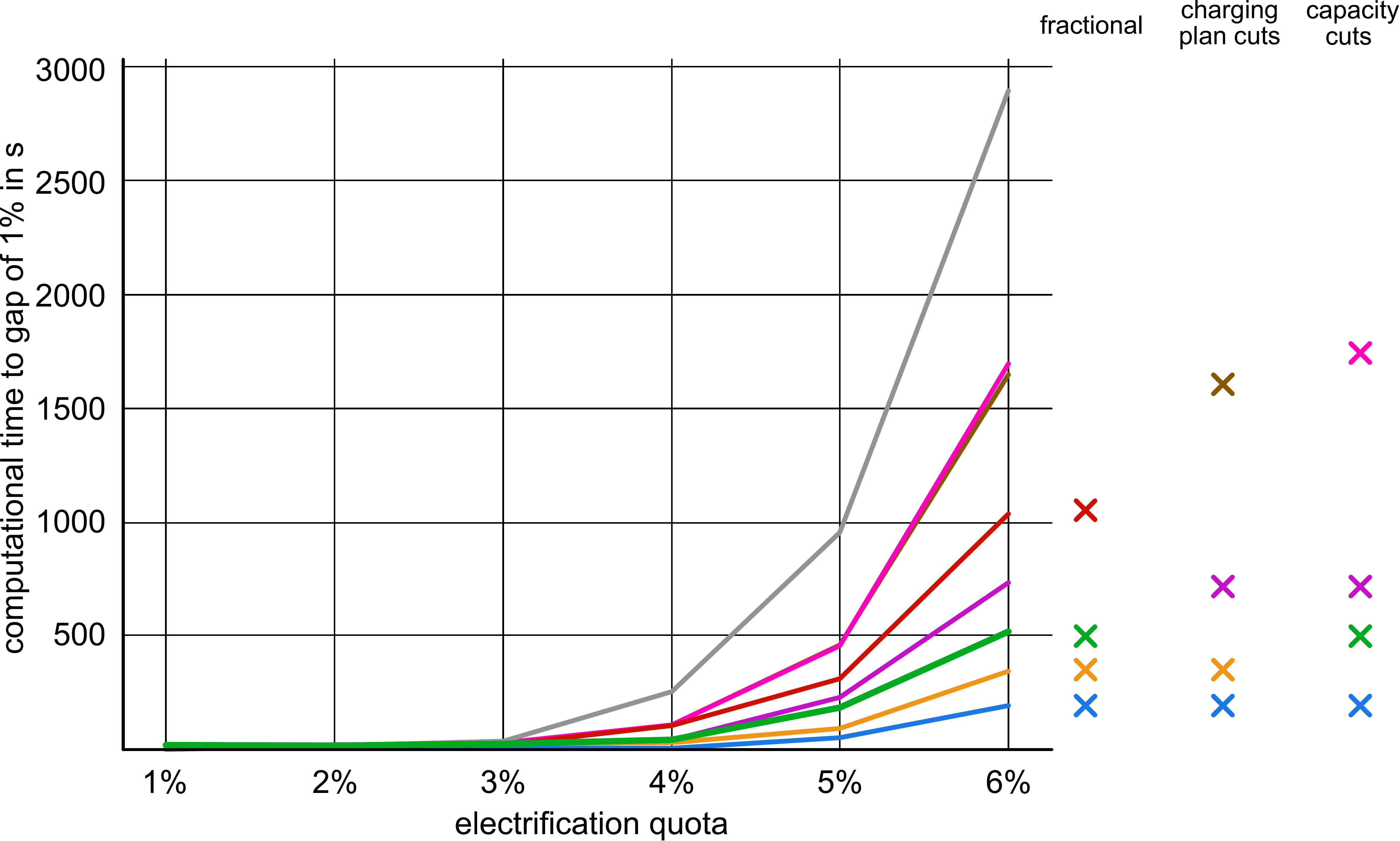}
    \end{center}
    \caption{Comparison of average computational times for enhanced models. Each curve represents a variant of our original model, where the selected enhancements are depicted on the right.}
    \label{fig:runtime}
\end{figure}

We observe that in our calculations, adding additional enhancements always has a positive effect on the computational time of our model. 
The strength of the enhancements appears to increase from `capacity cuts' to `charging plan cuts' to the fractional relaxation of the $y$ variables. 

Secondly, we evaluated the quality of the linear programming relaxations of the respective models by evaluating the `root LP gap' given by $(\mathrm{OPT}-D)/D$, where $\mathrm{OPT}$ denotes the optimum value and $D$ denotes the dual bound directly after the linear programming relaxation has been solved.
The results are depicted in \Cref{fig:gap}.

\begin{figure}
    \begin{center}
        \includegraphics[width=.8\textwidth]{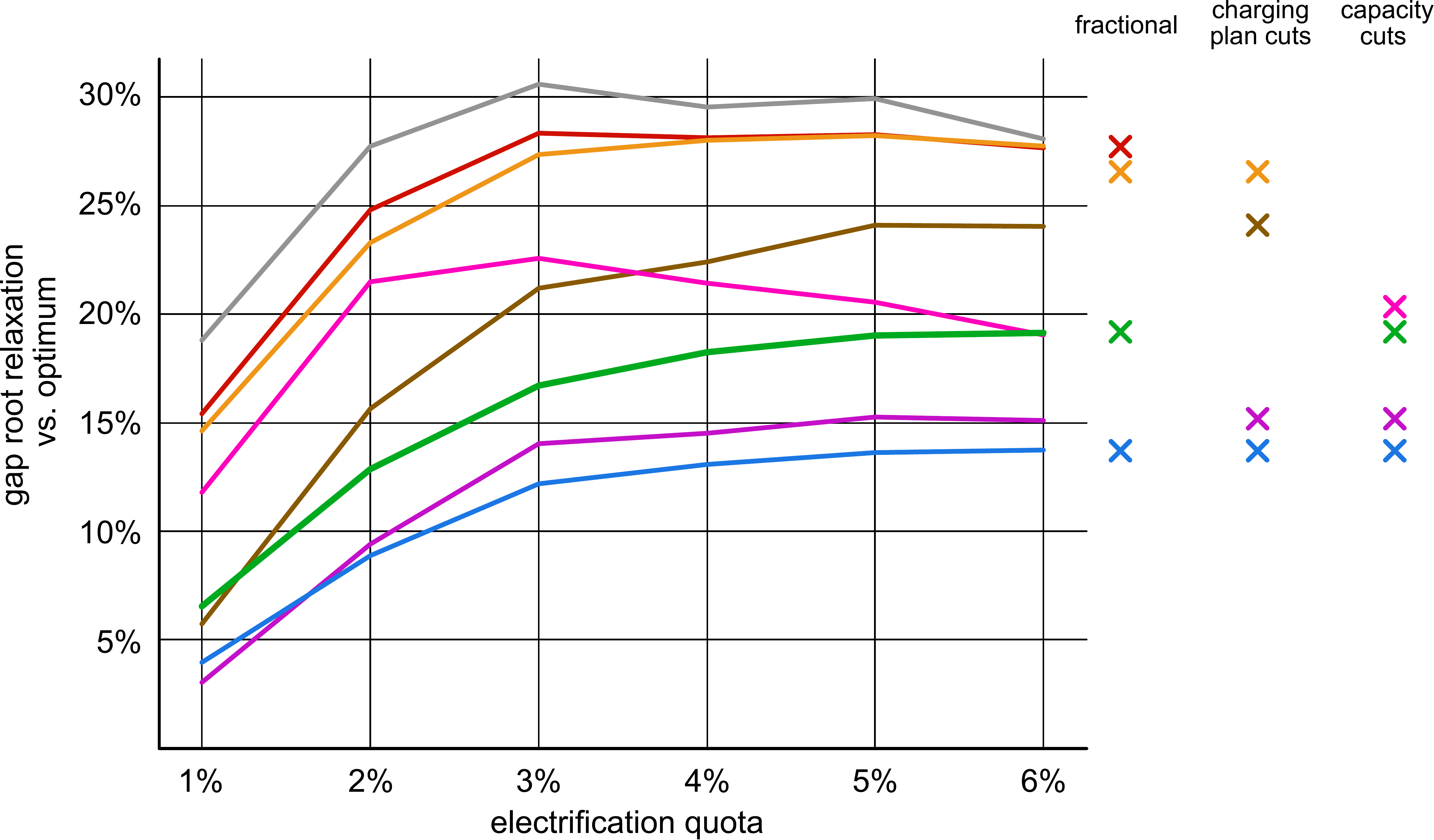}
    \end{center}
    \caption{Comparison of average root LP gap for enhanced models. Again, each curve represents a variant of our original model, where the selected enhancements are depicted on the right.}
    \label{fig:gap}
\end{figure}

It should be noted that since Gurobi's preprocessing routines may further strengthen the linear programming relaxations, it is difficult to specify the actual linear program that determines the dual bound $D$. 
However, as above, we observe that adding enhancements is favorable in order to improve the root LP gap.
In particular, adding the capacity cuts improves the gap most significantly.
Note that without preprocessing, the other two enhancements would not have an effect on the dual bound.

In summary, we recommend to use all three presented enhancements in a practical application.
For the `fractional' enhancement, we observe the largest improvement in terms of computational time, which comes at the cost that the solution is not necessarily feasible. 
Since the exact future mobility demand is not known a priori and the new solution value is almost indistinguishable from the original, this cost seems negligible.
The two different classes of cuts also improve the computational time by an extent that clearly outweighs their generation cost.

\subsection{Case study analysis}
In this section, we demonstrate that our approach is able to determine positions for charging infrastructure in large urban environments that satisfy driver demands to a high degree.

To this end, we consider the two following slight modifications of the presented model.
First, instead of minimizing the cost and satisfying all drivers, we maximize the number of satisfied drivers under the constraint that the cost is below a given budget.
This allows us to compare our solutions with existing positions, and can be easily incorporated into our model by introducing binary variables for each driver indicating whether their demands could be satisfied.
Second, instead of considering a single set of drivers, we incorporate multiple independent sets of drivers (which can be thought of as different days).
This avoids overfitting to driver data of a single day.

Notice that, in our model, not only the positioning of charging stations but also the allocation of a fixed set of drivers to charging stations is optimized.
Thus, in order to evaluate our method on realistic instances, it is necessary to understand how a solution performs on (i) an unknown set of drivers who (ii) occupy charging stations according to some natural behavior.

We propose to evaluate the positioning of charging stations by means of a driver-based simulation.
In Appendix~\ref{sec:simulation}, we describe a simple simulation, where each driver greedily aims at satisfying their own demand.
Recall that our goal is to determine charging stations that allow drivers to switch to \glspl{ev} without having to significantly deviate from their original schedules.
To this end, within the simulation, we count the number of drivers that need to take `large' detours in order to satisfy their demand.
Here, we say that a driver took a large detour if she charges at a distance of more than $\SI{400}{\metre}$ to the position of her respective break.
A walking distance in this range is a standard assumption in the related literature (cf.~\cite{schiffer_2021}).
If the distance is more than $\SI{5}{km}$, then we say that the driver's schedule is not compatible.

\begin{figure}
    \begin{center}
        \includegraphics[width=.6\textwidth]{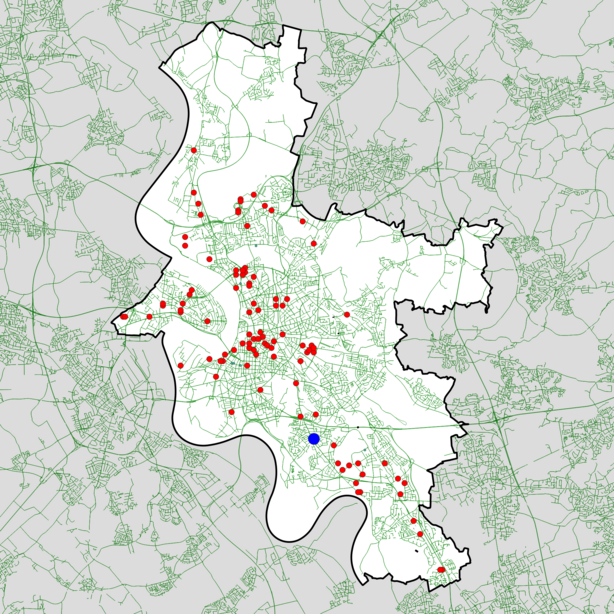}
    \end{center}
    \caption{Currently existing charging infrastructure in the city of Düsseldorf (November 2021). Red points represent AC charging stations, and blue points represent DC charging stations.}
    \label{fig:b1}
\end{figure}

The baseline for our experiments are the positions of existing charging infrastructure in Düsseldorf\footnote{\url{https://www.bundesnetzagentur.de/SharedDocs/Downloads/DE/Sachgebiete/Energie/Unternehmen_Institutionen/E_Mobilitaet/Ladesaeulenregister.xlsx}} (cf. \Cref{fig:b1}), which has a total cost of $C=272$ cost units as defined in \Cref{sec:case-study}.
We compute four solutions with budgets worth $C$, $2C$, $4C$ and $6C$, each based on two random driver sets corresponding to electrification rates of $3\%$, $5\%$, $12\%$, and $18\%$, respectively.
The specific models were chosen in a way that the respective budgets had to be fully used.
While there is no clear indication what electrification rates and number of driver sets are optimal as an input, we found that solutions obtained from different variations of the input performed almost equally in our evaluations.

\begin{figure}
    \begin{center}
        \begin{subfigure}[t]{.4\textwidth}
            \includegraphics[width=\textwidth]{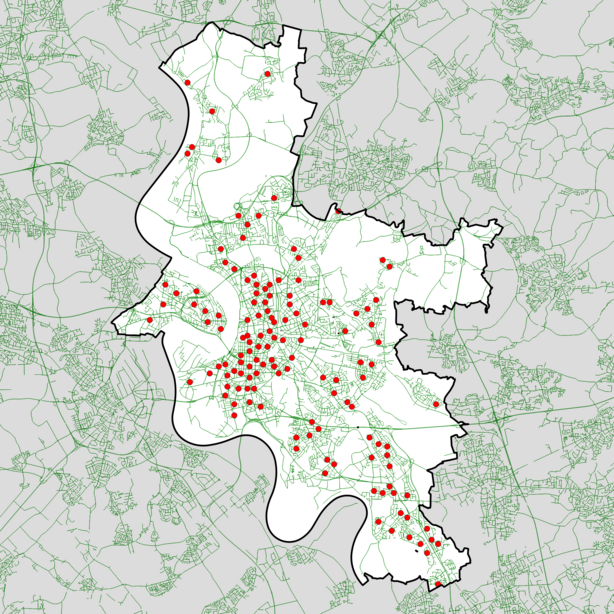}~
            \caption{Budget $C$}
            \label{fig:b2a}
        \end{subfigure}
        \begin{subfigure}[t]{.4\textwidth}
            \includegraphics[width=\textwidth]{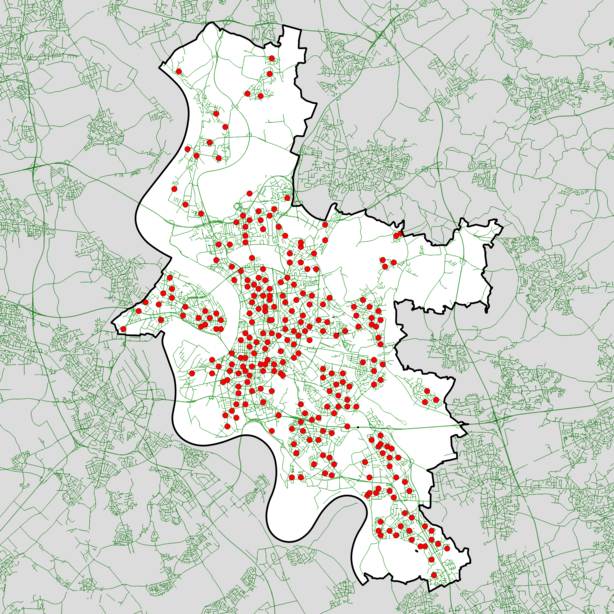}~
            \caption{Budget $2C$}
            \label{fig:b2b}
        \end{subfigure}
        \begin{subfigure}[t]{.4\textwidth}
            \includegraphics[width=\textwidth]{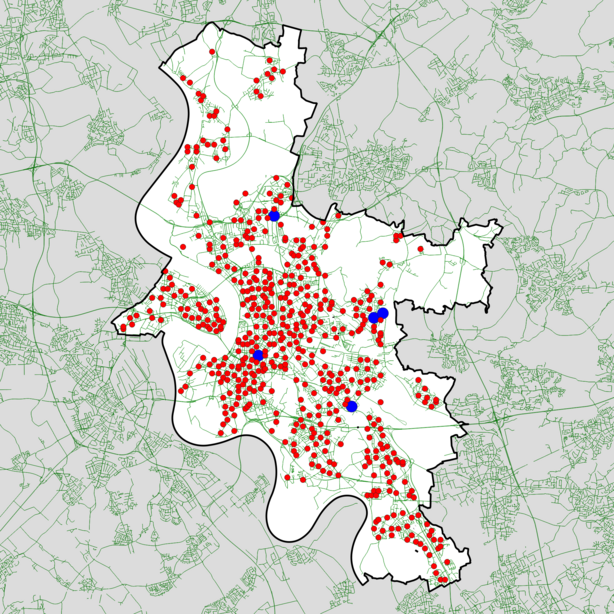}~
            \caption{Budget $4C$}
            \label{fig:b2c}
        \end{subfigure}
        \begin{subfigure}[t]{.4\textwidth}
            \includegraphics[width=\textwidth]{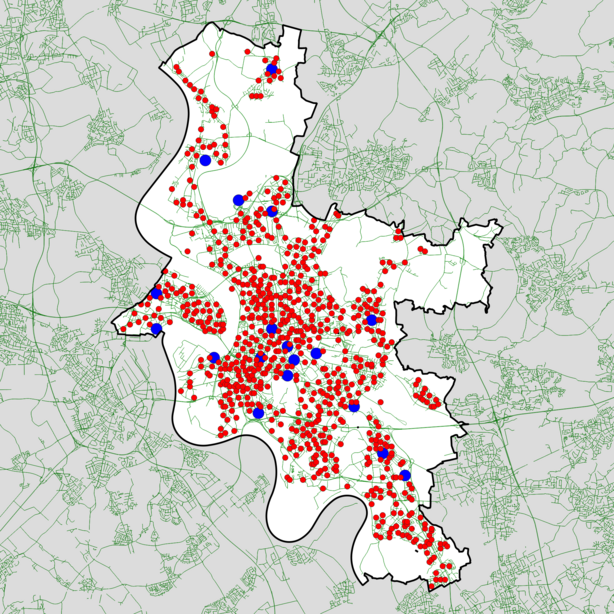}~
            \caption{Budget $6C$}
            \label{fig:b2d}
        \end{subfigure}
    \end{center}
    \caption{Charging station positions for selected solutions of the budget-constrained model. Red points represent AC charging stations, and blue points represent DC charging stations.}
    \label{fig:b2}
\end{figure}

Each solution was found within at most $\SI{5}{hrs}$ using Gurobi version 9.5.1 on a standard laptop and is within $1\%$ of the optimal solution.
The solutions consist of $272$, $544$, $1068$, and $1560$ charging ports, respectively, and the arrangements of the respective charging stations are shown in~\Cref{fig:b2}.
We see that, in each solution, the charging infrastructure is much more evenly dispersed than in the reference arrangement (see~\Cref{fig:b1}).
This is in particular observable in~\Cref{fig:b2a}, which operates on the same budget as the reference arrangement, but covers a larger area.
Since in our model, DC charging stations are twice as expensive as AC stations, they are only placed within scenarios that permit a higher budget, i.e., for a budget of more than $4C$ (see~\Cref{fig:b2c} and~\Cref{fig:b2d}).
The general structure of our solutions is an even distribution over the central part of the city (see in particular~\Cref{fig:b2a}).
Starting from a budget of $2C$, also outer areas are well covered (see~\Cref{fig:b2b}).
While a budget of $4C$ additionally allows for some DC charging stations (see~\Cref{fig:b2c}), a budget of $6C$ suffices for an even distribution of both types of charging stations over the whole city area.

\begin{figure}
    \begin{center}
        \includegraphics[width=.65\textwidth]{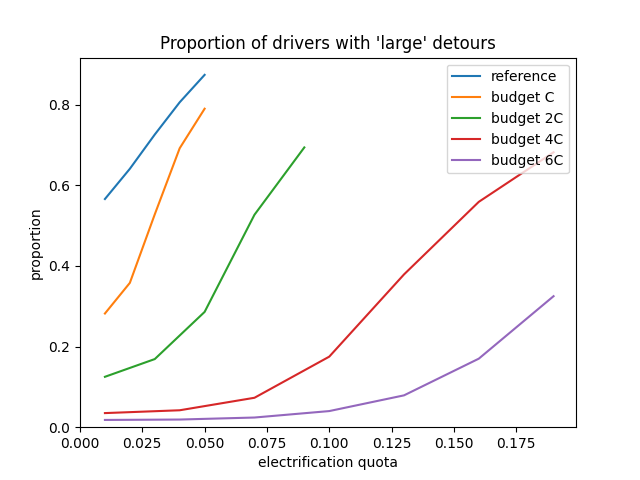}
    \end{center}
    \caption{Proportion of drivers that need to take a detour of more than $\SI{400}{\metre}$ for at least one of their charging processes. Each datapoint is the average over 10 simulation runs.}
    \label{fig:performance1}
\end{figure}

\begin{figure}
    \begin{center}
        \includegraphics[width=.65\textwidth]{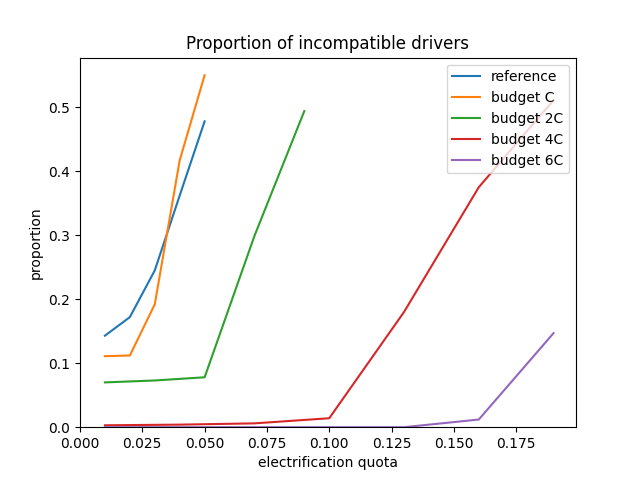}
    \end{center}
    \caption{Proportion of drivers that are incompatible. Each datapoint is the average over 10 simulation runs.}
    \label{fig:performance2}
\end{figure}

To evaluate the performance of our solutions, we depict the proportion of drivers that need to take at least one detour of at least $\SI{400}{\metre}$ in~\Cref{fig:performance1}.
We clearly see that our solution with a budget of $C$ outperforms the reference solution.
Still, neither of the two positionings can satisfy the demand of an unknown set of drivers to a satisfying degree.
In contrast, if we allow for a budget of $4C$, our solution provides sufficient charging infrastructure for populations that consist of up to $10\%$ of \glspl{ev}.
The 6-fold budget suffices for an electrification rate of $15\%$ consistently. 

In~\Cref{fig:performance2}, we show the proportion of drivers whose schedules are not compatible, i.e., who need to take a detour of more than $\SI{5\,000}{\metre}$ for at least one of their charging processes.
The graphic shows that the positionings of charging stations appear to have a `tipping point': a specific electrification rate at which the number of incompatible drivers begins to drastically increase.
At lower electrification rates, the proportion of incompatible drivers is almost constant, and is comprised of the drivers whose demands are not covered by the positioning.
This graphic underlines the result of \Cref{fig:performance1}:
while a budget of $2C$ is not sufficient to cover the entire region of the city (even for small electrification rates), our solution for the 4-fold budget performs well for electrification rates of up to $10\%$, and a 6-fold budget is sufficient for electrification rates of $15\%$.

\section{Conclusion}\label{sec:conclusion}

We introduced a mathematical programming based approach for determining positions of charging stations that allow \gls{ev} drivers to maintain their individual daily schedule without significant deviations for charging.
We formulated this problem as a combinatorial optimization problem and derived \gls{milp} models with strengthened formulations that allow to handle large instances and various objective functions.
We demonstrated how this approach can be utilized in practice based on traffic data containing mobility information of individual drivers.
We presented results of a case study for the city of Düsseldorf, which show that our approach allows to efficiently solve instances with electrification rates up to $15\%$.

Our work paves the way for future research from a methodological and from a practitioner's perspective.
From a methodological perspective, it remains interesting to enhance our framework for solving even larger instances.
This can be done by either further improving our exact approach or by developing competetive metaheuristics.
In the latter case, our exact approach may serve for benchmarking purposes.

From a practitioner's perspective, several questions remain that require additional case studies.
First, it remains interesting to analyze how existing charging infrastructure should be extended to meet increasing demand.
The approach presented in this paper is readily applicable to such a problem setting (by fixing decision variables for existing charging stations).
Second, analyzing various case studies can help to detect structural properties that yield profound insights to inform city planners in practice.
Third, extending our planning model to account for power network related constraints remains a challenging but also crucial avenue for future work.

\section*{Acknowledgements}
The work of the first author was funded by the Deutsche Forschungsgemeinschaft (German Research Foundation) under the project 277991500/GRK2201.
We would like to thank Holger Poppe for many stimulating discussions on this topic, the \emph{Transport Systems Planning and Transport Telematics group} of \emph{Technische Universität Berlin} for providing the data for our case study of Düsseldorf, and in particular Christian Rakow for answering questions around the data.

\bibliographystyle{plainnat}
\bibliography{paper}

\appendix
\section{Simulation}\label{sec:simulation}

In this section, we describe the simple driver-based simulation that has been used to evaluate the positioning of charging stations obtained by our main method.
The code for the simulation is available in our Github repository\footnote{\url{https://github.com/tumBAIS/driverAwareChargingInfrastructureDesign}} alongside with our optimization method.

Recall that our main approach is based on the idea that a set of charging stations performs well if most drivers are able to charge during breaks without deviating from their original schedules, making a transition to \glspl{ev} as easy as possible.
Given a fixed set of charging stations, we describe a simple simulation where each driver greedily aims at satisfying their own demand.
The number of drivers that need to take `large' detours in order to satisfy their demand is then regarded as a measure of performance.

\paragraph{Input}
Apart from a fixed set of charging stations to be evaluated, our simulation is again based on the set of drivers $\mathcal{D}$, which is obtained as described in \Cref{sec:case-study}.
Given an electrification rate, we sample a specific set of drivers $D$ and their starting \gls{soc} as explained in \Cref{sec:case-study}.
Moreover, we specify the maximum distance $r_h$ a driver is willing to walk between their stop and a charging station without deviating from the schedule, and a maximum distance $r_m$ that can be travelled in this manner at all.
For our experiments, we use $r_h:=400\mathrm{m}$ and $r_m:=5000\mathrm{m}$.

\paragraph{Main simulation loop}
The main attention of our simulation is devoted to the decisions of individual drivers in $D$.
We assume that each driver knows ahead of day about both the planned trips as well as the starting \gls{soc} of their \gls{ev}.
We say that a feasible charging plan is \emph{compatible} (to the set of charging stations) if for every charging stop there is a charging station with the desired mode within distance $r_m$, and initially label it as \emph{good} if the respective charging stations are within distance~$r_h$.
We assume that drivers prefer plans that both use few charging stops and have them as early as possible.
For every driver, we determine an initial \emph{preferred} charging plan accordingly from the set of good charging plans, or from the set of compatible charging plans if no good charging plan exists.

We sort the set of all breaks of drivers in $D$ in increasing order with respect to the starting time.
For each break of a driver $d$ in this list, we proceed as follows:
\begin{enumerate}
    \item If the current preferred plan of $d$ is good and contains a charging operation at the current break, we search for a charging station within distance $r_h$ that has a free port of the respective charging mode.
    \item If such a charging station exists, $d$ blocks a port for the whole duration of the break, and we proceed with the next break.
    \item If such a station does not exist, $d$ chooses a new good charging plan that still matches the drivers' charging operations so far and dismisses the current one, provided that such a new plan exists.
    \item Otherwise, if no good charging plans remain, we remove the label `good' from the current preferred charging plan and try to find a free charging station within a distance $r_m$.
    \item If no such station exists, dismiss the current plan and choose another compatible plan that still matches the drivers' charging operations so far.
    \item Only in the case, where no further compatible plan is available, and the current plan cannot be completed, we label a driver as ``not compatible'' and remove her breaks from the rest of the simulation.
\end{enumerate}

\paragraph{Evaluation}
Note that, at the end of the simulation, every driver that has not been labeled as ``not compatible'' was able to perform charging operations according to a feasible charging plan.
Moreover, if the preferred plan of a driver $d$ at the end of the simulation is still labeled as good, then $d$ was able to avoid large detours at all.
Thus, in order to evaluate how well the charging stations served the drivers' needs, we can compare the (i) number of drivers that were able to follow a feasible charging plan (i) without detours and (ii) with detours, and (iii) the number of drivers that were not able to follow a feasible charging plan at all (labeled as ``not compatible'').

\paragraph{Repeated simulation}
We remark that the initial distribution of the \glspl{soc} described above may seem arbitrary.
However, observe that at the end of each simulation loop, we can reconstruct the exact \gls{soc} of a driver at the end of the day by analyzing the performed charging processes throughout the day.
These \glspl{soc} can be used as new initial \glspl{soc} for another simulation loop.
Running the simulation loop several times yields \glspl{soc} that might serve as a more realistic input.

\end{document}